\documentclass[11pt,letterpaper]{amsart}

\usepackage[utf8]{inputenc}

\usepackage{amssymb}
\usepackage{amsmath}
\usepackage{amsthm}

\usepackage{tikz}
\usepackage{tikz-cd}
\usetikzlibrary{math}

\usepackage{enumitem}

\usepackage{mathtools}

\usepackage{longtable}

\usepackage{graphicx}

\usepackage{caption}
\usepackage{subcaption}
\usepackage[style=alphabetic,url=false]{biblatex}
\usepackage[english,noabbrev,sort]{cleveref}

\usepackage{my-math-definitions} 

\newtheorem{theorem}{Theorem}

\newtheorem{lemma}[theorem]{Lemma}
\newtheorem{corollary}[theorem]{Corollary}
\newtheorem{remark}[theorem]{Remark}
\newtheorem{question}[theorem]{Question}

\newtheorem{example}[theorem]{Example}

\DeclareMathOperator{\mult}{mult}

\DeclareMathOperator{\supp}{supp}
\DeclareMathOperator{\lm}{lm}
\DeclareMathOperator{\hull}{hull}
\DeclareMathOperator{\NP}{NP}
\DeclareMathOperator{\sm}{sm}

\newcommand{\I}{\mathcal I}

\author{Fabian Gundlach}
\address{
Harvard University,
Department of Mathematics,
1 Oxford Street,
Cambridge, MA 02138,
USA
}
\email{gundlach@math.harvard.edu}

\title{Polynomials vanishing at lattice points in a convex set}


\begin{document}

\begin{abstract}
Let $P$ be a bounded convex subset of $\mathbb R^n$ of positive volume. Denote the smallest degree of a polynomial $p(X_1,\dots,X_n)$ vanishing on $P\cap\mathbb Z^n$ by $r_P$ and denote the smallest number $u\geq0$ such that every function on $P\cap\mathbb Z^n$ can be interpolated by a polynomial of degree at most $u$ by $s_P$. We show that the values $(r_{d\cdot P}-1)/d$ and $s_{d\cdot P}/d$ for dilates $d\cdot P$ converge from below to some numbers $v_P,w_P>0$ as $d$ goes to infinity. The limits satisfy $v_P^{n-1}w_P \leq n!\cdot\operatorname{vol}(P)$. When $P$ is a triangle in the plane, we show equality: $v_Pw_P = 2\operatorname{vol}(P)$. These results are obtained by looking at the set of standard monomials of the vanishing ideal of $d\cdot P\cap\mathbb Z^n$ and by applying the Bernstein--Kushnirenko theorem. Finally, we study irreducible Laurent polynomials that vanish with large multiplicity at a point. This work is inspired by questions about Seshadri constants.
\end{abstract}

\maketitle

\tableofcontents

\section{Introduction}

We write $\N$ for the set of nonnegative integers. Let $n\geq1$.

Denote the Minkowski sum of sets $A,B\subseteq \R^n$ by
\[
A+B = \{a+b\mid a\in A,\ b\in B\}.
\]
For a subset $A\subseteq \R^n$, we let
\[
kA = A+\cdots+A = \{a_1+\cdots+a_k\mid a_1,\dots,a_k\in A\}
\]
for any integer $k\geq1$ and we let
\[
d\cdot A = \{d\cdot a\mid a\in A\}
\]
for any real number $d$.

We denote the set of functions $f:\R^n\ra \R$ with finite support $\supp(f)$ contained in $A\subseteq \R^n$ by $F_A$. For any subgroup $Z$ of $\R^n$, we make $F_Z$ a ring in which multiplication is the convolution operation
\[
(f\ast g)(c) = \sum_{\substack{a,b\in Z:\\ a+b=c}} f(a)g(b).
\]
Note that
\begin{equation}\label{suppast}
\supp(f\ast g)\subseteq\supp(f)+\supp(g).
\end{equation}
Denote the convex hull of a set $A\subseteq\R^n$ by $\hull(A)$. The \emph{Newton polytope} of a function $f\in F_{\R^n}$ is $\NP(f) = \hull(\supp(f))$. The Newton polytope of a convolution is
\begin{equation}\label{NPast}
\NP(f\ast g) = \NP(f) + \NP(g).
\end{equation}
A function $f$ is a unit in $F_Z$ if and only if the support of $f$ consists of exactly one point.

We denote the vanishing ideal of a subset $A$ of $\R^n$ by
\[
\I(A) \subseteq \R[X_1,\dots,X_n].
\]

For any $e=(e_1,\dots,e_n)\in\Z^n$ and any $a=(a_1,\dots,a_n)$, we will write
\[
a^e = a_1^{e_1}\cdots a_n^{e_n}.
\]
Furthermore, for any $e=(e_1,\dots,e_n)\in\N^n$, we let
\[
|e| = e_1 + \cdots + e_n
\qquad\textnormal{ and }\qquad
e! = e_1!\cdots e_n!
\]
so that
\[
(a+b)^e = \sum_{\substack{p,q\in\N^n:\\ p+q=e}} \frac{e!}{p!q!}\cdot a^p b^q.
\]
For $p,q\in\R_{\geq0}^n$, we write $p\preceq q$ if $p_i\leq q_i$ for all $i=1,\dots,n$. We call $A$ a \emph{lower subset} of $\R_{\geq0}^n$ or $\N^n$ if for all $q\in A$ and for each $p\preceq q$ in $\R^n_{\geq0}$ or $\N^n$, respectively, we have $p\in A$.

Fix a total monomial well-order $\leq$ on the set of monomials in $X_1,\dots,X_n$. We denote the leading monomial of a nonzero polynomial $p$ by $\lm(p)$. For any polynomial ideal $I$, we write
\[
\lm(I) = \{\lm(p) \mid 0\neq p\in I\}.
\]
The monomials $X^e\notin\lm(I)$ are also called the \emph{standard monomials} of $I$. They form a basis of the vector space $\R[X_1,\dots,X_n]/I$. To any finite set $A\subseteq \R^n$, we associate the set
\[
E_A = \{e\in\N^n \mid X^e \notin\lm(\I(A))\}.
\]
The sets $E_A$ satisfy the following properties:
\begin{enumerate}[label=\alph*)]
\item $E_A$ is a lower subset of $\N^n$. 
\item If $A\subseteq B$, then $E_A\subseteq E_B$.
\item The monomials $X^e$ with $e\in E_A$ form a basis for the set of functions $A\ra \R$. In particular,
$|E_A| = |A|$.
\item $E_{A+v}=E_A$ for any $v\in \R^n$ because $p(X)$ and $p(X+v)$ have the same leading monomial.
\end{enumerate}

The dimension of the vector space of polynomials $p\in\I(A)$ with $p=0$ or $\lm(p)<X^e$ is the number of $e'\in\N^n\setminus E_A$ with $X^{e'}<X^e$. In particular:

\begin{remark}\label{deg_E}
If $\leq$ is a monomial order such that $X^e<X^{e'}$ whenever $|e|<|e'|$, the dimension of the vector space of polynomials $f\in\I(A)$ of degree at most $d$ is the number of $e\in\N^n\setminus E_A$ such that $|e|\leq d$.
\end{remark}

The Combinatorial Nullstellensatz (see \cite{alon-combinatorial-nullstellensatz}) computes $E_A$ when $A$ is a Cartesian product of finite sets:

\begin{lemma}\label{combinatorial_nullstellensatz}
Let $A_1,\dots,A_n$ be finite subsets of $\R$. Then,
\[
E_{A_1\times\cdots\times A_n} = \{0,\dots,|A_1|-1\}\times\cdots\times\{0,\dots,|A_n|-1\}.
\]
\end{lemma}
\begin{proof}
For any index $i$, the polynomial $\prod_{t\in A_i}(X_i-t)$ with leading monomial $X_i^{|A_i|}$ vanishes on $A_1\times\cdots\times A_n$. Therefore, $(0,\dots,0,|A_i|,0,\dots,0)\notin E_{A_1\times\dots\times A_n}$. Because $E_{A_1\times\cdots\times A_n}$ is a lower subset of $\N^n$, this implies that
\[
E_{A_1\times\cdots\times A_n} \subseteq \{0,\dots,|A_1|-1\}\times\cdots\times\{0,\dots,|A_n|-1\}.
\]
Equality follows because both sides have size $|A_1|\cdots|A_n|$.
\end{proof}

There are other cases in which $E_A$ can easily be computed. For example:

\begin{corollary}\label{lower}
Any finite lower subset $A$ of $\N^n\subset\R^n$ satisfies $E_A = A$.
\end{corollary}
\begin{proof}
For any $e=(e_1,\dots,e_n)\in A$, the set $A$ contains
\[B_e=\{0,\dots,e_1\}\times\cdots\times\{0,\dots,e_n\}.
\]
By the previous lemma, this implies that $E_A$ also contains $B_e$, so in particular $e\in E_A$. Hence, $A\subseteq E_A$. Equality follows because both sides have the same size.
\end{proof}

For a lexicographic monomial order, the set $E_A$ can be determined inductively, as described in \cite{lex-game}. We will come back to this in \cref{lex_section}.

Understanding the sets $E_A$ for other monomial orders can be more difficult, although $E_A$ can of course be computed for any particular set $A$ and monomial order $\leq$. (See for example \cite{computing-groebner-bases-of-finite-sets} for an algorithm for computing~$E_A$.)

For any $f\in F_{\R^n}$ and any polynomial $p(X_1,\dots,X_n)$, we write
\[
\langle f,p\rangle
= \sum_{a\in \R^n} f(a) p(a).
\]
If $f\neq0$, let $X^{\sm(f)}$ be the smallest monomial such that $\langle f,X^{\sm(f)}\rangle\neq0$. (Such a monomial exists because any finitely supported function can be interpolated by a polynomial. There is a smallest such monomial because any monomial order is a well-order.)

We will often work with the following description of $E_A$:

\begin{lemma}
For any finite set $A\subseteq \R^n$, we have
\[
E_A = \{\sm(f) \mid 0\neq f\in F_A\}.
\]
\end{lemma}
\begin{proof}
We have $e\in E_A$ if and only if there is no polynomial $p\in\mathcal I(A)$ with $\lm(p)=X^e$. Equivalently, the monomial $X^e$ is not a linear combination of smaller monomials when considered a function $A\ra \R$. This is equivalent to the existence of a function $f:A\ra \R$ orthogonal to all smaller monomials, but not to $X^e$, or in other words a function $f\neq 0$ with $\sm(f)=X^e$.
\end{proof}

For any $f,g\in F_{\R^n}$ and any $e\in\N^n$, we have
\begin{align}\label{conv_prod}
\langle f\ast g, X^e\rangle
&= \sum_{\substack{p,q\in\N^n:\\ p+q=e}} \frac{e!}{p!q!} \cdot \langle f, X^p\rangle \cdot \langle g,X^q\rangle. 
\end{align}
Hence,
\[
\langle f\ast g,X^e\rangle = 0
\]
for all monomials $X^e<X^{\sm(f)+\sm(g)}$ and
\[
\langle f\ast g,X^{\sm(f)+\sm(g)}\rangle = \frac{(\sm(f)+\sm(g))!}{\sm(f)!\sm(g)!} \cdot \langle f,X^{\sm(f)}\rangle \cdot \langle g,X^{\sm(g)}\rangle.
\]
It follows that
\begin{equation}\label{smast}
\sm(f\ast g) = \sm(f) + \sm(g).
\end{equation}
Together with (\ref{suppast}), we arrive at the crucial observation that
\begin{equation}\label{East}
E_A+E_B \subseteq E_{A+B}
\end{equation}
for any two finite sets $A,B\subseteq \R^n$.

We now give an overview of the remainder of this note. For simplicity, we assume that $\leq$ is a degree-lexicographic monomial order in this overview (and most of the paper), although most results can easily be generalized to arbitrary monomial orders.

In \cref{convex_section}, we consider dilates $d\cdot P$ of a bounded convex subset $P$ of $\R^n$ of positive volume and study the asymptotic behavior of the set $E_{d\cdot P\cap\Z^n}$ as $d$ goes to infinity. Using (\ref{East}) together with the fact $|E_A|=|A|$, we show that there is a convex set $S_P\subseteq\R_{\geq0}^n$ of the same volume as $P$ such that (in some precise sense)
\begin{equation}\label{approximation}
E_{d\cdot P\cap\Z^n} \approx d\cdot S_P\cap\N^n.
\end{equation}
In \Cref{vw_converge}, we prove the results on general bounded convex sets $P$ mentioned in the abstract.

The ring $F_{\Z^n}$ is naturally isomorphic to the ring of Laurent polynomials in $n$ variables. In \cref{laurent_section}, we interpret $|\sm(f)|$ as the order of vanishing of the Laurent polynomial corresponding to $f$ at the point $(1,\dots,1)$. The Seshadri constant of a weighted projective plane turns out to be related to the set $S_P$ for a triangle~$P$. (This connection with Seshadri constants originally motivated our work.)

In \cref{relatively_prime_section}, we prove a lower bound for the mixed volume of Newton polytopes of relatively prime elements $f_1$ and $f_2$ of $F_{\Z^2}$:
\[
|\sm(f_1)||\sm(f_2)| \leq 2\vol(\NP(f_1), \NP(f_2)).
\]

This inequality is used in \cref{triangle_section} to improve the results from \cref{convex_section} when $P$ is a triangle.

Given the above inequality, it is natural to study the set of irreducible functions $f\in F_{\Z^2}$ with $|\sm(f)|>\sqrt{2\vol(\NP(f))}$. In \cref{irreducible_section}, we give an infinite family of such functions, as well as a list of such functions with small Newton polytope.

\textbf{Acknowledgments.} The author is very grateful to Ziquan Zhuang for introducing him to Seshadri constants and for helpful discussions.

\section{Lattice points in convex sets}\label{convex_section}

For any bounded set $P\subseteq\R^n$, we will consider the finite set $P\cap\Z^n\subseteq\R^n$.

If $P,Q\subseteq\R^n$ are bounded sets, then (\ref{East}) implies that
\[
E_{P\cap\Z^n} + E_{Q\cap\Z^n} \subseteq E_{P\cap\Z^n + Q\cap\Z^n} \subseteq E_{(P+Q)\cap\Z^n}.
\]

In particular, if $P$ is convex, then by induction
\[
d\cdot E_{P\cap\Z^n} \subseteq d E_{P\cap\Z^n} \subseteq E_{dP\cap\Z^n} = E_{d\cdot P\cap\Z^n}
\]
for any integer $d\geq1$. We obtain the following ascending chain of sets:
\[
E_{P\cap\Z^n} \subseteq 2^{-1}\cdot E_{2\cdot P\cap\Z^n} \subseteq 2^{-2}\cdot E_{2^2\cdot P\cap\Z^n} \subseteq \cdots
\]
Let
\[
S_P = \overline{\bigcup_{k\geq0} 2^{-k}\cdot E_{2^k\cdot P\cap\Z^n}} \subseteq \R_{\geq0}^n,
\]
be the closure of their union.

\begin{example}
For the triangle $P=\hull(\{(0,0),(4,2),(2,3)\})$, \Cref{triangle_approx} shows the first four sets $2^{-k}\cdot E_{2^k\cdot P\cap\Z^n}$ along with the set $S_P$ (which we will compute in \Cref{example_quad}).
\end{example}

\begin{figure}[ht]
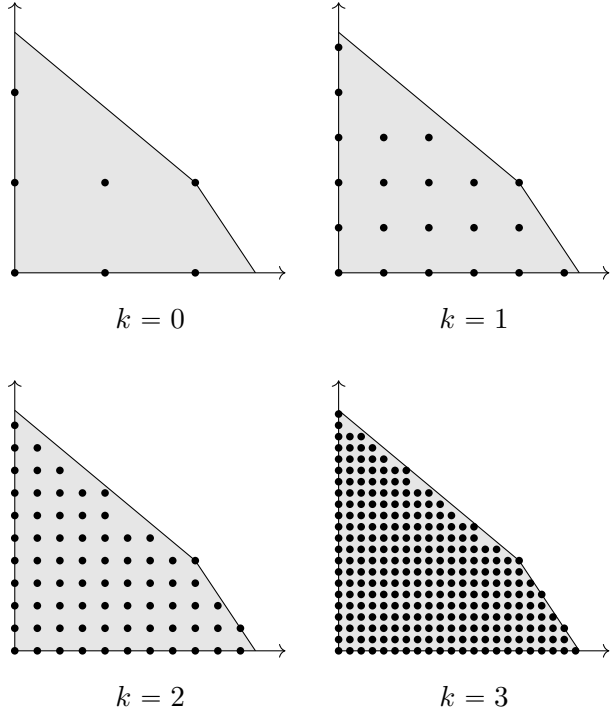

\newcommand{\pic}[2]{
\tikzpicture[radius=1.2pt,scale=1.2]
\fill[black!10] (0,0) -- (8/3,0) -- (2,1) -- (0,8/3) -- cycle;
\draw[->] (0,0) -- (3,0);
\draw[->] (0,0) -- (0,3);
\input{E_#1.tex}
\draw (8/3,0) -- (2,1) -- (0,8/3);
\draw (1.5,-0.5) node {$k=#2$};
\endtikzpicture
}

\tikzpicture[radius=1.2pt,scale=1.2]
\fill[black!10] (0,0) -- (8/3,0) -- (2,1) -- (0,8/3) -- cycle;
\draw[->] (0,0) -- (3,0);
\draw[->] (0,0) -- (0,3);
\input{E_1.tex}
\draw (8/3,0) -- (2,1) -- (0,8/3);
\draw (1.5,-0.5) node {$k=0$};
\endtikzpicture
\hspace{10pt}

\tikzpicture[radius=1.2pt,scale=1.2]
\fill[black!10] (0,0) -- (8/3,0) -- (2,1) -- (0,8/3) -- cycle;
\draw[->] (0,0) -- (3,0);
\draw[->] (0,0) -- (0,3);
\input{E_2.tex}
\draw (8/3,0) -- (2,1) -- (0,8/3);
\draw (1.5,-0.5) node {$k=1$};
\endtikzpicture
\\[15pt]

\tikzpicture[radius=1.2pt,scale=1.2]
\fill[black!10] (0,0) -- (8/3,0) -- (2,1) -- (0,8/3) -- cycle;
\draw[->] (0,0) -- (3,0);
\draw[->] (0,0) -- (0,3);
\input{E_4.tex}
\draw (8/3,0) -- (2,1) -- (0,8/3);
\draw (1.5,-0.5) node {$k=2$};
\endtikzpicture
\hspace{10pt}

\tikzpicture[radius=1.2pt,scale=1.2]
\fill[black!10] (0,0) -- (8/3,0) -- (2,1) -- (0,8/3) -- cycle;
\draw[->] (0,0) -- (3,0);
\draw[->] (0,0) -- (0,3);
\input{E_8.tex}
\draw (8/3,0) -- (2,1) -- (0,8/3);
\draw (1.5,-0.5) node {$k=3$};
\endtikzpicture

\caption{The first four sets $2^{-k}\cdot E_{2^k\cdot P\cap\Z^n}$ (black dots) and the set $S_P$ (gray)}\label{triangle_approx}
\end{figure}

The sets $S_P$ inherit many properties from $E_A$. For example:
\begin{enumerate}[label=\alph*)]
\item $S_P$ is a lower subset of $\R_{\geq0}^n$.
\item If $P\subseteq Q$, then $S_P\subseteq S_Q$.
\item $S_P + S_Q \subseteq S_{P+Q}$.
\end{enumerate}

\begin{example}\label{cont_comb_nsts}
Let $P_1,\dots,P_n\subseteq\R$ be intervals of lengths $l_1,\dots,l_n>0$. Then, the Combinatorial Nullstellensatz shows that
\[
S_{P_1\times\cdots\times P_n} = [0,l_1]\times\cdots\times[0,l_n].
\]
\end{example}
\begin{example}\label{cont_lower}
More generally, any compact convex lower subset $P$ of $\R_{\geq0}^n$ satisfies $S_P=P$ by \Cref{lower}.
\end{example}

\begin{theorem}\label{cont_basic}
For any bounded convex set $P$ of positive volume:
\begin{enumerate}[label=\alph*)]
\item $S_P$ is convex.
\item $S_P$ is compact.
\item $\vol(S_P) = \vol(P)$.
\item $S_{d\cdot P} = d\cdot S_P$ for all real numbers $d>0$.\label{Sscaling}
\item $S_{P+v}=S_P$ for all $v\in\R^n$.
\end{enumerate}
\end{theorem}
\begin{proof}
\begin{enumerate}[label=\alph*)]
\item Note that $S_P$ contains the midpoint of any two points in $S_P$:
\[S_P + S_P \subseteq S_{P+P} = S_{2\cdot P} = 2\cdot S_P.
\]
Since $S_P$ is closed, it is therefore convex.
\item Since $P$ is contained in an axis-parallel box, \Cref{cont_comb_nsts} shows that so is $S_P$. Hence, $S_P$ is bounded and in fact compact.
\item For any $k\geq0$, we have $2^{-k}\cdot E_{2^k\cdot P\cap\Z^n} \subseteq S_P\cap\N^n$ and therefore
\[
|2^k\cdot P\cap\Z^n| = |E_{2^k\cdot P\cap\Z^n}| \leq |2^k\cdot S_P\cap\N^n|.
\]
Since $P$ is bounded and convex, the left-hand side is
\[
|2^k\cdot P\cap\Z^n| = 2^{nk}\cdot\vol(P) + \O_P(2^{(n-1)k}).
\]
Since $S_P$ is also bounded and convex, the right-hand side is
\[
|2^k\cdot S_P\cap\N^n| = 2^{nk}\cdot\vol(S_P) + \O_P(2^{(n-1)k}).
\]
Hence, $\vol(P)\leq\vol(S_P)$.

For a finite lower subset $L$ of $\N^n$, let $B(L)$ be the smallest lower subset of $\R_{\geq0}^n$ containing $L$. It's easy to see that
\[\vol(B(L)) = |\{(e_1,\dots,e_n)\in L\mid e_1,\dots,e_n\geq1\}| \leq |L|.
\]
Now, $S_P$ is the closure of the union of the sets
\[
B(E_{P\cap\Z^n}) \subseteq 2^{-1}\cdot B(E_{2\cdot P\cap\Z^n}) \subseteq 2^{-2}\cdot B(E_{2^2\cdot P\cap\Z^n}) \subseteq \cdots,
\]
which have volume
\begin{align*}
&\vol(2^{-k}\cdot B(E_{2^k\cdot P\cap\Z^n})) \\
&= 2^{-nk}\cdot \vol(B(E_{2^k\cdot P\cap\Z^n})) \\
&\leq 2^{-nk}\cdot |E_{2^k\cdot P\cap\Z^n}| \\
&= 2^{-nk}\cdot |2^k\cdot P\cap\Z^n| \\
&= 2^{-nk}\cdot (2^{nk}\cdot\vol(P) + \O_P(2^{(n-1)k})) \\
&= \vol(P) + \O_P(2^{-k}).
\end{align*}
As these sets form an ascending chain, letting $k$ go to infinity, we see that
\[
\vol(2^{-k}\cdot B(E_{2^k\cdot P\cap\Z^n}))
\leq \vol(P)
\]
and then
\[
\vol(S_P) \leq \vol(P).
\qedhere
\]
\item We first show the equality for integers $d\geq1$. To this end, it suffices to note that $d\cdot S_P \subseteq d S_P\subseteq S_{dP}=S_{d\cdot P}$ and equality follows because both sides are compact convex sets of the same positive volume $d^n\cdot\vol(P)$.

Equality for all rational numbers $d>0$ follows immediately.

For irrational numbers, we make use of a monotonicity argument: As $P$ has positive volume and therefore contains an open ball, there is a number $T>0$ such that for all $t>T$, we have $t\cdot P\cap\Z^n\neq\emptyset$. Let $0<d_1<d_2$. For all $k\geq0$ sufficiently large that $2^k(d_2-d_1) > T$, there is a vector $v\in 2^k(d_2-d_1)\cdot P\cap\Z^n$. Then,
\begin{align*}
E_{2^kd_1\cdot P\cap\Z^n}
&= E_{2^kd_1\cdot P\cap\Z^n + v}
= E_{(2^kd_1\cdot P+v)\cap\Z^n} \\
&\subseteq E_{(2^kd_1\cdot P+2^k(d_2-d_1)\cdot P)\cap\Z^n}
= E_{2^kd_2\cdot P\cap\Z^n}.
\end{align*}
Hence, $S_{d_1\cdot P}\subseteq S_{d_2\cdot P}$ whenever $0<d_1<d_2$. If $d>0$ is any real number, we have shown for each rational number $0<d'<d$ that $d'\cdot S_P = S_{d'\cdot P} \subseteq S_{d\cdot P}$. Since $S_{d\cdot P}$ is closed, this implies that $d\cdot S_P\subseteq S_{d\cdot P}$. The opposite inclusion follows by replacing $d$ and $P$ by $d^{-1}$ and $d\cdot P$.
\item If $v\in\Z^n$, then $E_{2^k\cdot P\cap\Z^n}=E_{2^k\cdot P\cap\Z^n+2^k\cdot v}=E_{2^k\cdot(P+v)\cap\Z^n}$ for all $k\geq0$, so $S_P=S_{P+v}$. For arbitrary $v\in\R^n$, equality then follows using~\ref{Sscaling} and the fact that $\R\otimes_\Z\Z^n=\R^n$.
\qedhere
\end{enumerate}
\end{proof}

The following is a precise form of the claimed approximation (\ref{approximation}). 

\begin{theorem}\label{cont_converges}
For any $0<\varepsilon\leq1$, there is a number $D>0$ such that for all real numbers $d>D$, we have
\[
(1-\varepsilon)d\cdot S_P\cap\N^n \subseteq E_{d\cdot P\cap\Z^n} \subseteq d\cdot S_P\cap\N^n.
\]
\end{theorem}
\begin{proof}
We have
\[
E_{d\cdot P\cap\Z^n} \subseteq S_{d\cdot P} \cap \N^n = d\cdot S_P \cap \N^n
\]
for all $d>0$, so it only remains to prove the first containment.

If $e\in (1-\varepsilon)d\cdot S_P\cap\N^n$ doesn't lie in the lower subset $E_{d\cdot P\cap\Z^n}$ of $\N^n$, then
\begin{align*}
E_{d\cdot P\cap\Z^n}
&\subseteq S_{d\cdot P} \cap \N^n \setminus (e+\N^n) \\
&= d\cdot S_P \cap \N^n \setminus (e+\N^n) \\
&= d\cdot (S_P \setminus (\tfrac1d\cdot e+\R_{\geq0}^n)) \cap \N^n.
\end{align*}
For large $d$, the left-hand side has size
\[
|d\cdot P\cap\Z^n| = d^n\cdot\vol(P) + \O_P(d^{n-1})
\]
and the right-hand side has size
\[
d^n \cdot (\vol(S_P) - \vol(S_P \cap (\tfrac1d\cdot e+\R_{\geq0}^n))) + \O_P(d^{n-1}).
\]
As $\vol(S_P)=\vol(P)$, we conclude that
\[
\vol(S_P \cap (\tfrac1d\cdot e+\R_{\geq0}^n)) = \O_P(d^{-1}).
\]
But since $\tfrac1d\cdot e\in (1-\varepsilon)\cdot S_P$ and the set $S_P\subseteq\R_{\geq0}^n$ is convex,
\[
S_P \cap (\tfrac1d\cdot e+\R_{\geq0}^n) \supseteq \tfrac1d\cdot e + \varepsilon\cdot S_P
\]
and in particular
\[
\vol(S_P \cap (\tfrac1d\cdot e+\R_{\geq0}^n)) \geq \vol(\varepsilon\cdot S_P)>0.
\]
For large enough $d$, we arrive at a contradiction.
\end{proof}

\begin{corollary}\label{vw_converge}
For any bounded convex set $P\subseteq\R^n$ of positive volume, let $r_P$ be the smallest degree of a polynomial vanishing on $P\cap\Z^n$ and let $s_P$ be the smallest number $u\geq0$ such that every function on $P\cap\Z^n$ can be interpolated by a polynomial of degree at most $u$. Then, $(r_{d\cdot P}-1)/d$ and $s_{d\cdot P}/d$ converge from below to some numbers $v_P,w_P>0$ as $d$ goes to infinity. These numbers satisfy $v_P^{n-1} w_P \leq n!\cdot \vol(P)$.
\end{corollary}
\begin{proof}
Let $\leq$ be a degree-lexicographic monomial order. Then, $r_P$ is the smallest degree $|e|$ of a monomial $X^e\in\lm(\mathcal I(P\cap\Z^n))$, or equivalently the smallest value $|e|$ for $e\notin E_{P\cap\Z^n}$. In other words, $r_P$ is the largest number such that
\[
\{e\in\N^n \mid |e|\leq r_P-1\} \subseteq E_{P\cap\Z^n}.
\]
By \Cref{cont_basic} and \Cref{cont_converges}, we see that $(r_{d\cdot P}-1)/d$ converges from below to the largest number $v_P$ such that
\[
\{e\in\R_{\geq0}^n \mid |e|\leq v_P\} \subseteq S_P.
\]
Similarly, $s_P$ is the largest number such that $E_{P\cap\Z^n}$ contains an element $e$ with $|e|=s_P$. Then, $S_{d\cdot P}/d$ converges from below to the largest number $w_P$ such that $S_P$ contains an element $e$ with $|e|=w_P$.

Now, $S_P$ contains the convex hull of some points
\[
(0,\dots,0),(v_P,0,\dots,0),(0,v_P,0,\dots,0),\dots,(0,\dots,0,v_P),(e_1,\dots,e_n)
\]
with $|e|=w_P\geq v_P$. (See \Cref{shapes}.) This convex hull has volume $\frac1{n!}\cdot v_P^{n-1}w_P$. Since $S_P$ has volume $\vol(P)$, we indeed have $v_P^{n-1}w_P\leq n!\cdot\vol(P)$.
\end{proof}

\begin{figure}[ht]
\begin{subfigure}[b]{0.45\textwidth}
\begin{tikzpicture}[radius=4/3pt,scale=1.4]
\fill[black!10] (0,0) -- (1.7,0) -- (1.4,1.1) -- (0,1.7) -- cycle;
\draw[->] (0,0) -- (3,0);
\draw[->] (0,0) -- (0,3);
\draw (1.7,0) -- (1.4,1.1) -- (0,1.7);
\fill (1.4,1.1) circle node[above right] {$(e_1,e_2)$};
\fill (1.7,0) circle node[below] {$(v,0)$};
\fill (0,1.7) circle node[left] {$(0,v)$};
\fill (0,0) circle;
\draw[dashed] (2.5,0) -- (0,2.5);
\draw[dashed] (1.7,0) -- (0,1.7);
\end{tikzpicture}
\end{subfigure}
\begin{subfigure}[b]{0.45\textwidth}
\begin{tikzpicture}[radius=4/3pt,scale=1.4]
\fill[black!10] (0,0) -- (1.7,0) -- (0,2.5) -- cycle;
\draw[->] (0,0) -- (3,0);
\draw[->] (0,0) -- (0,3);
\draw (1.7,0) -- (0,2.5);
\fill (0,0) circle;
\fill (1.7,0) circle node[below] {$(v,0)$};
\fill (0,2.5) circle node[left] {$(e_1,e_2)$};
\draw[dashed] (2.5,0) -- (0,2.5);
\draw[dashed] (1.7,0) -- (0,1.7);
\end{tikzpicture}
\end{subfigure}
\caption{Possible shapes of the lower bound for $S_P$}\label{shapes}
\end{figure}
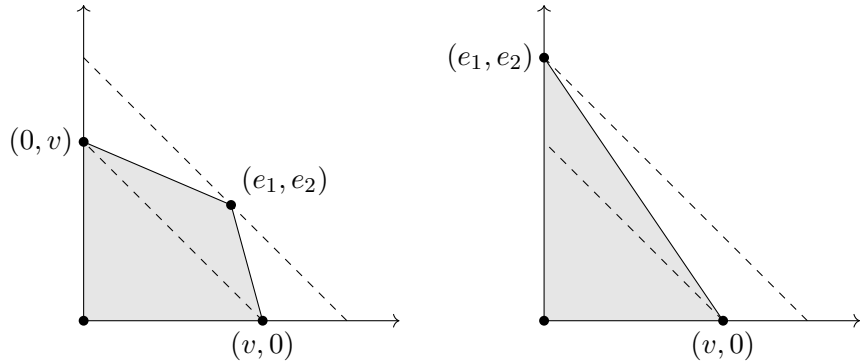

For lexicographic monomial orders, the set $S_P$ can again be determined inductively. (See \Cref{contlex} in \cref{lex_section}.)

Computing $S_P$ for other monomial orders can be more difficult. At least, $S_P$ (and therefore $v_P$ and $w_P$) can be approximated to any precision:

\begin{remark}
For any $d>0$, we have
\begin{equation}\label{approx_bounds}
A_{P,d} \subseteq S_P \subseteq B_{P,d},
\end{equation}
where
\[
A_{P,d} = \hull(d^{-1}\cdot E_{d\cdot P\cap\Z^n})
\]
and
\[
B_{P,d} = \{e\in\R_{\geq0}^n \mid \vol(\hull(A_{P,d}\cup\{e\})) \leq \vol(P)\}.
\]
The approximations $A_{P,d}$ and $B_{P,d}$ converge to $S_P$ in the following sense: For any $\varepsilon>0$, there is a number $D>0$ such that for all real numbers $d>D$, we have
\begin{equation}\label{approx_conv_lower}
(1-\varepsilon)\cdot S_P \subseteq A_{P,d}
\end{equation}
and
\begin{equation}\label{approx_conv_upper}
B_{P,d} \subseteq (1+\varepsilon)\cdot S_P.
\end{equation}
\end{remark}
\begin{proof}
Firstly, (\ref{approx_bounds}) is clear: We have already shown the first inclusion, and the second inclusion follows since $S_P$ is convex and $\vol(S_P)=\vol(P)$.

We have shown (\ref{approx_conv_lower}) in \Cref{cont_converges}. Then, (\ref{approx_conv_upper}) follows from (\ref{approx_conv_lower}) and a simple compactness argument.
\end{proof}

\section{Lexicographic monomial orders}\label{lex_section}

This section can be skipped by readers interested only in degree-le\-xi\-co\-graphic mo\-no\-mi\-al orders.

As explained in \cite{lex-game}, the set $E_A$ can be determined inductively for any lexicographic monomial order:

\begin{theorem}\label{discretelex}
Assume that $\leq$ is the lexicographic monomial order with $X_1<\cdots<X_n$. Let $A\subseteq \R^n$ be a finite set. For any $a_1\in \R$, consider the fiber
\[
B(a_1) = \{(a_2,\dots,a_n) \mid (a_1,a_2\dots,a_n)\in A\} \subseteq \R^{n-1}.
\]
We have $e=(e_1,\dots,e_n)\in E_A$ if and only if
\begin{equation}\label{lexeq}
|\{a_1\in \R\mid (e_2,\dots,e_n) \in E_{B(a_1)}\}| > e_1,
\end{equation}
where we use the ``restricted'' lexicographic monomial order with $X_2<\cdots<X_n$ on $\R^{n-1}$ to define $E_{B(a_1)}$.
\end{theorem}
\begin{proof}
Let $E_A'$ be the set of $e\in\N^n$ satisfying (\ref{lexeq}). Let $e\in E_A'$ and choose distinct $t_1,\dots,t_m\in \R$ such that $(e_2,\dots,e_n)\in E_{B(t_i)}$ for $i=1,\dots,m$, with $m>e_1$. Assume $e\notin E_A$, so there is a polynomial $p\in \R[X_1,\dots,X_n]$ with $\lm(p)=X^e=X_1^{e_1}\cdots X_n^{e_n}$ that vanishes at every point in $A$. Since $m>e_1$, the $X_2^{e_2}\cdots X_n^{e_n}$-coefficient of $p(t_i,X_2,\dots,X_n)$ is nonzero for some $i$. Since we are considering the lexicographic monomial order with $X_1<\cdots<X_n$, it follows that $X_2^{e_2}\cdots X_n^{e_n}$ is the leading monomial of $p(t_i,X_2,\dots,X_n)$. By definition, since $(e_2,\dots,e_n)\in E_{B(t_i)}$, this polynomial cannot vanish on $B(t_i)$, so $p$ cannot vanish on $\{t_i\}\times B(t_i)\subseteq A$. This is a contradiction, so $E_A'\subseteq E_A$.

It only remains to show that both sets have the same size:
\begin{align*}
|E_A'|
&= \sum_{(e_2,\dots,e_n)\in\N^{n-1}} |\{a_1\in \R \mid (e_2,\dots,e_n)\in E_{B(a_1)}\}| \\
&= \sum_{a_1\in \R} |E_{B(a_1)}| = \sum_{a_1\in \R} |B(a_1)| = |A| = |E_A|.
\qedhere
\end{align*}
\end{proof}

We can similarly determine the sets $S_P$ for any lexicographic monomial order:

\begin{theorem}\label{contlex}
Assume that $\leq$ is the lexicographic monomial order with $X_1<\cdots<X_n$. Let $P\subseteq\R^n$ be a bounded convex set of positive volume. For any $v_1\in\R$, consider the fiber
\[
Q(v_1) = \{(v_2,\dots,v_n) \mid (v_1,v_2\dots,v_n)\in P\} \subseteq\R^{n-1}.
\]
We have $e=(e_1,\dots,e_n)\in S_P$ if and only if $G_{e_2,\dots,e_n}\neq\emptyset$ and $\vol(G_{e_2,\dots,e_n})\geq e_1$, where
\begin{equation}\label{contlexeq}
G_{e_2,\dots,e_n}=\{v_1\in\R\mid \vol(Q(v_1))>0 \textnormal{ and } (e_2,\dots,e_n) \in S_{Q(v_1)}\}
\end{equation}
and we use the ``restricted'' lexicographic monomial order with $X_2<\cdots<X_n$ on $\R^{n-1}$ to define $S_{Q(v_1)}$.
\end{theorem}
\begin{proof}
The set
\[
G := \{(v_1,e_2,\dots,e_n) \mid \vol(Q(v_1))>0\textnormal{ and }(e_2,\dots,e_n)\in S_{Q(v_1)}\}
\]
is convex: For any $r,s\in\R$ and $0\leq t\leq1$, the convexity of $P$ implies that
\[
t\cdot Q(r)+(1-t)\cdot Q(s) \subseteq Q(tr+(1-t)s).
\]
Hence, if $\vol(Q(r))>0$ and $\vol(Q(s))>0$, then $\vol(Q(tr+(1-t)s))>0$ and
\[
t\cdot S_{Q(r)}+(1-t)\cdot S_{Q(s)} \subseteq S_{Q(tr+(1-t)s)}.
\]
Therefore, $G$ is indeed convex. It is clearly bounded. In particular, each horizontal fiber $G_{e_2,\dots,e_n}$ is an interval. Let $S_P'$ be the set of $e\in\R_{\geq0}^n$ with $G_{e_2,\dots,e_n}\neq\emptyset$ and $\vol(G_{e_2,\dots,e_n})\geq e_1$. It is a bounded convex lower subset of $\R_{\geq0}^n$. Note that $S_{d\cdot P}' = d\cdot S_P'$ for all real numbers $d>0$.

It's not hard to show that since $G$ is convex and each vertical fiber $S_{Q(v_1)}$ of $G$ is closed, 
\[
\ol{G_{e_2,\dots,e_n}} = \{v_1\in\R \mid (v_1,e_2,\dots,e_n)\in \ol{G}\}.
\]
This implies that $S_P'$ is closed.

Let $a$ and $b$ be the infimum and supremum, respectively, of the $x_1$-co\-or\-di\-nates of points $(x_1,\dots,x_n)\in P$. Since $P$ is convex and has positive volume, we have $\vol(Q(v_1)) > 0$ for all $a<v_1<b$. Replacing $P$ by $d\cdot P$ for some real number $d>0$, we can assume without loss of generality that $a$ and $b$ are irrational.

We now check that $E_{2^k\cdot P\cap\Z^n}\subseteq S_{2^k\cdot P}' = 2^k\cdot S_P'$ for all $k\geq0$, and therefore $S_P\subseteq S_P'$. Without loss of generality, we can assume $k=0$. Let $e\in E_{P\cap\Z^n}$. According to \Cref{discretelex}, the set
\[
J_{e_2,\dots,e_n} := \{v_1\in\Z \mid (e_2,\dots,e_n) \in E_{B(v_1)}\}
\]
with $B(v_1) = Q(v_1)\cap\Z^{n-1}$ has size at least $e_1+1$.

On the other hand, $J_{e_2,\dots,e_n}$ is contained in the interval $G_{e_2,\dots,e_n}$: Firstly, we have $\vol(Q(v_1))>0$ for all $v_1\in J_{e_2,\dots,e_n}$ as the only possible elements $v_1\in\R$ with $\vol(Q(v_1))=0$ but $Q(v_1)\neq\emptyset$ are $a$ and $b$, neither of which is an integer. Secondly, $(e_2,\dots,e_n)\in S_{Q(v_1)}$ because $E_{B(v_1)} = E_{Q(v_1)\cap\Z^{n-1}} \subseteq S_{Q(v_1)}$.

Therefore, the interval $G_{e_2,\dots,e_n}$ contains $e_1+1$ integers, so it must have length at least $e_1$. Hence,  indeed $e\in S_P'$.

We have now shown that $S_P\subseteq S_P'$. Both sides are compact convex sets of the same positive volume:
\begin{align*}
\vol(S_P')
&= \int_{\R_{\geq0}^{n-1}} \vol(\{a<v_1<b \mid (e_2,\dots,e_n)\in S_{Q(v_1)}\}) \dd(e_2,\dots,e_n) \\
&= \int_{\R} \vol(S_{Q(v_1)}) \dd v_1 = \int_{\R} \vol(Q(v_1)) \dd v_1 = \vol(P) \\
&= \vol(S_P).
\qedhere
\end{align*}
\end{proof}

\begin{remark}\label{other_lattices}
Most results in this note can be extended to lattices $\Lambda\subseteq\R^n$ of rank $n$ other than $\Lambda=\Z^n$, with one caveat: Let $n=2$ and replace $\Z^2$ by the lattice $\Lambda\subseteq\R^2$ spanned by $(1,0)$ and $(\sqrt2,1)$. Assume $\leq$ is the lexicographic monomial order with $X_1<X_2$. Then $S_P=\R_{\geq0}\times\{0\}$ is unbounded for all bounded convex sets $P\subseteq\R^2$ of positive volume: Since no two points in $d\cdot P\cap\Z^2\subseteq\Z^2$ have the same $X_1$-coordinate, \Cref{discretelex} tells us that
\[E_{d\cdot P\cap\Z^2}=\{0,\dots,|d\cdot P\cap\Z^2|-1\}\times\{0\}.
\]
The size $|d\cdot P\cap\Z^n|=d^2\cdot\vol(P)+\O_P(d)$ grows quadratically in~$d$, so $S_P=\R_{\geq0}\times\{0\}$ is unbounded. However, almost all results in this note can be adapted for example when the monomial order is such that there are only finitely many monomials less than any given monomial.
\end{remark}

\section{Laurent polynomials}\label{laurent_section}

The ring $F_{\Z^n}$ is naturally isomorphic to the ring of Laurent polynomials $\R[X_1^{\pm1},\dots,X_n^{\pm1}]$. Here, a function $f\in F_{\Z^n}$ corresponds to the Laurent polynomial
\[
\pi_f = \sum_{a\in\Z^n} f(a) X^a.
\]

Let $\leq$ be a degree-lexicographic monomial order.

\begin{lemma}
Let $0\neq f\in F_{\Z^2}$. Then, $|\sm(f)|$ is the order of vanishing of $\pi_f$ at the point $(1,\dots,1)\in \R^n$.
\end{lemma}
\begin{proof}
Denote the derivative operator with respect to $X_i$ by $\partial_i$ and consider the operator $\delta_i = X_i\partial_i$. For any $e=(e_1,\dots,e_n)\in\N^n$, write $\partial^e=\partial_1^{e_1}\cdots\partial_n^{e_n}$ and $\delta^e=\delta_1^{e_1}\cdots\delta_n^{e_n}$. The order of vanishing of $\pi_f$ at $(1,\dots,1)$ is the largest number $d\geq0$ such that $\partial^e\pi_f(1,\dots,1)=0$ for all $e\in\N^n$ with $|e|<d$. On the other hand, $|\sm(f)|$ is the largest number $d\geq0$ such that $\langle f,X^e\rangle=0$ for all $e\in\N^n$ with $|e|<d$. For any $e\in\N^n$ and $a\in\Z^n$, we have
\[
\delta^e X^a = (\delta_1^{e_1}X_1^{a_1}) \cdots (\delta_n^{e_n}X_n^{a_n}) = (a_1^{e_1}X_1^{a_1}) \cdots (a_n^{e_n}X_n^{e_n}) = a^e X^a.
\]
Hence,
\[
\delta^e\pi_f = \sum_{a\in\Z^n} f(a) \delta^e X^a = \sum_{a\in\Z^n} f(a) a^e X^a,
\]
so $\delta^e\pi_f(1,\dots,1) = \sum_{a\in\Z^n} f(a) a^e = \langle f, X^e\rangle$. Thus, $|\sm(f)|$ is the largest number $d\geq0$ such that $\delta^e\pi_f(1,\dots,1)=0$ for all $e\in\N^n$ with $|e|<d$. By induction, any $\delta^e$ can be written as
\[
\delta^e = X^e\partial^e + \sum_{e\neq e'\preceq e} c_{e,e'} X^{e'} \partial^{e'}
\]
with constants $c_{e,e'}$. Hence, we have $\partial^e\pi_f(1,\dots,1)=0$ for all $e\in\N^n$ with $|e|<d$ if and only if $\delta^e\pi_f(1,\dots,1)=0$ for all $e\in\N^n$ with $|e|<d$.
\end{proof}

In particular, for any bounded convex set $A\subseteq\R^n$, the maximal order of vanishing at $(1,\dots,1)$ of a nonzero Laurent polynomial whose Newton polytope is contained in $A$ is $\max\{|e|\mid e\in E_{A\cap\Z^n}\}$.

Let $a,b,c\geq1$ be relatively prime integers. The \emph{Seshadri constant} of the weighted projective plane $\vP^{a,b,c}$ at the generic point $x=[1:1:1]\in\vP^{a,b,c}$ is
\[
\inf\bigg\{\frac{\deg(C)}{\mult_x(C)} \ \bigg|\ C\subset\vP^{a,b,c}\textnormal{ curve through $x$}\bigg\}.
\]
See \cite{bauer-primer} for an introduction to Seshadri constants.
A curve $C$ of degree $d$ is given by a polynomial equation $f(A,B,C)=0$, where each monomial $A^iB^jC^k$ in $f$ has weighted degree $ai+bj+ck = d$. The multiplicity $\mult_x(C)$ of $C$ at the point $x=[1:1:1]$ is the largest number $m\geq0$ such that $\partial_A^i\partial_B^j f(1,1,1) = 0$ for all $i,j\geq0$ with $i+j<m$.

Let $P'\subset\R^3$ be the triangle spanned by $(1/a,0,0)$, $(0,1/b,0)$, $(0,0,1/c)$. Let
\[
H = \{(x,y,z)\in\R^3 \mid ax+by+cz=1\}
\]
be the hyperplane containing these three points. For any integer $d\geq0$, the points $(i,j,k)$ in $d\cdot P'\cap\Z^3$ correspond to the monomials $A^iB^jC^k$ of degree $ai+bj+ck=d$. Let $\rho:\R^3\ra\R^2$ be a linear map sending the (translated) lattice $H\cap\Z^3$ of rank two to the lattice $\Z^2$. Let $P=\rho(P')$. For any integer $d\geq0$, this map $\rho$ induces a bijection between $d\cdot P'\cap\Z^3$ and $d\cdot P\cap\Z^2$.

It follows that curves $C$ of degree $d$ with $\mult_x(C)=m$ are in bijection with functions $f\in F_{d\cdot P\cap\Z^2}$ (up to multiplication by a scalar) such that $|\sm(f)| = m$.

In particular, the Seshadri constant of $\vP^{a,b,c}$ is $w_P^{-1}$.

\section{Relatively prime functions}\label{relatively_prime_section}

In this section, let $n=2$. Although the results can easily be extended to any non-lexicographic monomial order, we assume for simplicity that $\leq$ is the degree-lexicographic monomial order with $X_1<X_2$. For any $e\in\N^2$, the number of monomials less than $X^e$, which we denote by $q(e)$ is
\[
q(e) = \binom{|e|+1}{2} + e_2.
\]
For any compact convex sets $P,Q\subset\R^2$, the \emph{mixed volume} $\vol(P,Q)$ is defined by
\[
\vol(P+Q)=\vol(P)+\vol(Q)+2\vol(P,Q).
\]
In this section, we prove the following theorem. We remind the reader that $F_{\Z^2}$ is isomorphic to a ring of Laurent polynomials, an in particular a unique factorization domain.

\begin{theorem}\label{relatively_prime_inequality}
If $f_1,f_2$ are relatively prime elements of $F_{\Z^2}$, then
\[
|\sm(f_1)||\sm(f_2)| \leq 2\vol(\NP(f_1),\NP(f_2)).
\]
\end{theorem}

The idea behind this inequality is simple: The Bernstein--Kushnirenko theorem (see \cite{bernstein-bkk} or \cite[section 7.5]{using-algebraic-geometry}) implies that the intersection number of the Laurent polynomials $\pi_{f_1}$ and $\pi_{f_2}$ is $2\vol(\NP(f_1),\NP(f_2))$. On the other hand, the Laurent polynomials already vanish to order $|\sm(f_1)|$ and $|\sm(f_2)|$, respectively, at the point $(1,1)$, so their intersection number is at least $|\sm(f_1)||\sm(f_2)|$. (The same idea is also used for Seshadri constants for example in Lemma~5.2 of \cite{bauer-seshadri-constants}.)

Monotonicity of mixed volumes then implies:

\begin{corollary}\label{relatively_prime_inequality_same}
If $f_1,f_2$ are relatively prime elements of $F_{\Z^2}$ whose supports are contained in the convex set $P$, then
\[
|\sm(f_1)||\sm(f_2)| \leq 2\vol(P).
\]
\end{corollary}

We call an irreducible element $f$ of $F_{\Z^2}$ \emph{large} if $|\sm(f)| > \sqrt{2\vol(\NP(f))}$. According to \Cref{relatively_prime_inequality}, there cannot be more than one large element~$f$ (up to units) with the same Newton polytope.

On the other hand, if $f_1$ and $f_2$ are non-large irreducible elements of $F_{\Z^2}$, then the inequality in \Cref{relatively_prime_inequality} immediately follows from the Brunn--Minkowski inequality $\vol(P_1,P_2)\geq\sqrt{\vol(P_1)\vol(P_2)}$.

We now give another more self-contained proof of \Cref{relatively_prime_inequality}, both for the sake of completeness and because it can more easily be generalized to other monomial orders. We adapt here the proof of Bézout's theorem in section~5.2 of \cite{fulton} and of the fact that the intersection multiplicity of two curves at a point is at least the product of the multiplicities of those curves at the point in section~3.3 of \cite{fulton}.

For any $e\in\N^2$ and any compact convex set $P$, let $W_{P,e}\subseteq F_{P\cap\Z^2}$ be the vector space of functions $f\in F_{P\cap\Z^2}$ such that $\langle f,X^{e'}\rangle=0$ for all monomials $X^{e'}<X^e$. By (\ref{suppast}) and (\ref{smast}), we have
\[
W_{P_1,e_1}\ast W_{P_2,e_2} \subseteq W_{P_1+P_2,e_1+e_2}.
\]
Let $H_{P,e} = F_{P\cap\Z^2}/W_{P,e}$. Now, $H_{P,e}$ is dual to the vector space of functions on $P\cap\Z^2$ that can be interpolated by a polynomial $p$ with $p=0$ or $\lm(p)<X^e$. The dimension of this space is
\[
\dim(H_{P,e}) = |\{e'\in E_{P\cap\Z^2} \mid X^{e'}<X^e\}|.
\]
Therefore, for any $e\in\N^2$, if $P$ contains a sufficiently large axis-parallel square, \Cref{combinatorial_nullstellensatz} implies that
\[
\dim(H_{P,e}) = q(e).
\]

We will use the following consequence of Pick's theorem:

\begin{lemma}\label{minkowski_count}
Let $P_1,P_2$ be compact convex lattice polygons (polygons whose corners lie in $\Z^2$). Then,
\[
|(P_1+P_2)\cap\Z^2| = |P_1\cap\Z^2| + |P_2\cap\Z^2| + 2\vol(P_1,P_2) - 1.
\]
\end{lemma}
\begin{proof}
Pick's theorem states that any compact lattice polygon $P$ satisfies
\[
|P\cap\Z^2| = \vol(P) + 1 + \frac{B(P)}{2},
\]
where $B(P)$ is the number of lattice points on the edges of $P$. (If $P$ is a line segment, we count points in the interior of this line segment twice, thinking of $P$ as having two edges with opposite orientation. If $P$ is a single point, $B(P)=0$.) The edges of the Minkowski sum $P_1+P_2$ are (translates of) the edges of $P_1$ and $P_2$ (joining edges with the same orientation). Hence, $B(P_1+P_2)=B(P_1)+B(P_2)$. Then,
\begin{align*}
&|(P_1+P_2)\cap\Z^2| \\
&= \vol(P_1+P_2) + 1 + \frac{B(P_1+P_2)}2 \\
&= \vol(P_1)+\vol(P_2)+2\vol(P_1,P_2) + 1 + \frac{B(P_1)+B(P_2)}2 \\
&= |P_1\cap\Z^2| + |P_2\cap\Z^2| + 2\vol(P_1,P_2) - 1.
\qedhere
\end{align*}
\end{proof}

\begin{corollary}\label{parallelogram_law}
Let $P_1,P_2,U$ be compact convex lattice polygons. Then,
\[
|(U+P_1+P_2)\cap\Z^2| - |(U+P_1)\cap\Z^2| - |(U+P_2)\cap\Z^2| + |U\cap\Z^2|
= 2\vol(P_1,P_2).
\]
\end{corollary}
\begin{proof}
We have
\begin{align*}
& |(U+P_1+P_2)\cap\Z^2| + |U\cap\Z^2| \\
&= |(U+P_1+P_2+U)\cap\Z^2| - 2\vol(U+P_1+P_2,U) + 1
\end{align*}
and
\begin{align*}
& |(U+P_1)\cap\Z^2| + |(U+P_2)\cap\Z^2| \\
&= |(U+P_1+U+P_2)\cap\Z^2| - 2\vol(U+P_1,U+P_2) + 1.
\end{align*}
Moreover,
\[
\vol(U+P_1,U+P_2) - \vol(U+P_1+P_2,U) = \vol(P_1,P_2)
\]
because mixed volumes are bilinear.
\end{proof}

\begin{proof}[Proof of \Cref{relatively_prime_inequality}]
Write $P_i=\NP(f_i)$ and $e_i=(e_{i1},e_{i2})=\sm(f_i)\in\N^2$.

Let $U$ be a compact convex lattice polygon. Consider the following sequence of linear maps:
\[
0\ra F_{U\cap\Z^2} \stackrel{\alpha_U}\ra F_{(U+P_2)\cap\Z^2} \times F_{(U+P_1)\cap\Z^2} \stackrel{\beta_U}\ra F_{(U+P_1+P_2)\cap\Z^2},
\]
where the first map sends $b$ to $(b\ast f_2,-b\ast f_1)$ and the second map sends $(a_1,a_2)$ to $a_1\ast f_1+a_2\ast f_2$. Because $f_1$ and $f_2$ are relatively prime, the kernel of $\beta_U$ consists exactly of the pairs of the form $(b\ast f_2,-b\ast f_1)$ with $b\in F_{\Z^2}$. Since $\NP(b\ast f_2)=\NP(b)+\NP(f_2)$, we have $b\ast f_2\in F_{(U+P_2)\cap\Z^2}$ if and only if $b\in F_{U\cap\Z^2}$. Hence, the above sequence is exact, so
\begin{align*}
&\dim(\coker(\beta_U)) \\
&= \dim(F_{(U+P_1+P_2)\cap\Z^2}) - \dim(F_{(U+P_2)\cap\Z^2}) - \dim(F_{(U+P_1)\cap\Z^2}) + \dim(F_{U\cap\Z^2}) \\
&= |(U+P_1+P_2)\cap\Z^2| - |(U+P_2)\cap\Z^2| - |(U+P_1)\cap\Z^2| + |U\cap\Z^2| \\
&= 2\vol(P_1,P_2).
\end{align*}
By definition, $f_i\in W_{P_i,e_i}$. We therefore have a well-defined linear map
\[
\gamma_U: H_{U+P_2,e_2} \times H_{U+P_1,e_1} \ra H_{U+P_1+P_2,e_1+e_2}
\]
sending $(a_1,a_2)$ to $a_1\ast f_1+a_2\ast f_2$. There is a natural (surjective) quotient map $\coker(\beta_U)\ra\coker(\gamma_U)$. Hence,
\begin{align*}
2\vol(P_1,P_2)
&= \dim(\coker(\beta_U)) \\
&\geq \dim(\coker(\gamma_U)) \\
&\geq \dim(H_{U+P_1+P_2,e_1+e_2}) - \dim(H_{U+P_2,e_2}) - \dim(H_{U+P_1,e_1}).
\end{align*}
If $U$ contains a sufficiently large square, the right-hand side is
\begin{align*}
& q(e_1+e_2) - q(e_1) - q(e_2) \\
&= \binom{|e_1|+|e_2|+1}{2} - \binom{|e_1|+1}{2} - \binom{|e_2|+1}{2} \\
&= |e_1||e_2|.
\qedhere
\end{align*}
\end{proof}

\section{Lattice points in triangles}\label{triangle_section}

In this section, we again let $n=2$ and assume for simplicity that $\leq$ is the degree-lexicographic monomial order with $X_1<X_2$. We will study the set $S_P$ when $P$ is a triangle.

For a compact convex set $P\subseteq\R^n$ of positive volume and a nonempty compact set $A\subseteq\R^n$, we let $l_{P,A}=l_{P,\hull(A)}$ be the smallest number $l\geq0$ such that $A\subseteq l\cdot P+v$ for some $v\in\R^n$.

\begin{lemma}\label{largest_triangle}
Let $P\subset\R^2$ be a compact triangle and let $A,B\subset\R^2$ be nonempty finite sets. Then, $l_{P,A+B}=l_{P,A}+l_{P,B}$.
\end{lemma}

This lemma fails for polygons with more than three vertices: For example, take $P=[0,1]\times[0,1]$ and $A=\{(0,0),(1,0)\}$ and $B=\{(0,0),(0,1)\}$. Then, $l_{P,A}=l_{P,B}=l_{P,A+B}=1$.

\begin{proof}
Applying an affine linear transformation, we can assume without loss of generality that $P$ is the standard triangle
\[
P = \{(x,y)\in\R^2 \mid x,y\geq0,\ x+y\leq 1\}.
\]
Then,
\[
l_{P,A} = \max_{(x,y)\in A}(x+y) - \min_{(x,y)\in A}x - \min_{(x,y)\in A}y
\]
for any finite set $A\subset\R^2$ and the claim follows immediately.
\end{proof}

For a nonzero function $f\in F_{\Z^2}$, we will also write $l_{P,f} = l_{P,\supp(f)} = l_{P,\NP(f)}$.

When computing Seshadri constants, it suffices to consider irreducible curves. More generally:

\begin{theorem}\label{triangle_spanned}
Let $P\subset\R^2$ be a compact triangle. Let
\[
Z_P = \{(0,0)\} \cup \bigg\{\frac{\sm(f)}{l_{P,f}} \ \bigg|\ f\in F_{\Z^2}\textnormal{ irreducible}\bigg\}.
\]
Then, $S_P = \overline{\hull(Z_P)}$.
\end{theorem}
\begin{proof}
\begin{description}
\item[``$\supseteq$''] Since $S_P$ is convex and closed, it suffices to show that $S_P \supseteq Z_P$. Let $f\in F_{\Z^2}$ be irreducible. There is a vector $v\in\R^2$ such that $Q:=\NP(f)\subseteq l_{P,f}\cdot P+v$. We have
\[
\sm(f)\in E_{Q\cap\Z^2}\subseteq S_Q = S_{Q-v}
\]
and therefore
\[
\frac{\sm(f)}{l_{P,f}} \in S_{l_{P,f}^{-1}\cdot(Q-v)} \subseteq S_P.
\]
\item[``$\subseteq$''] Let $(0,0)\neq e\in S_P$. Let $U\subseteq\R^2$ be an open neighborhood of $e$ not containing $(0,0)$. By the definition of $S_P$, there is an integer $d\geq1$ (in fact a power of two) such that $U\cap d^{-1}\cdot E_{d\cdot P\cap\Z^2}\neq\emptyset$. Hence, there is a nonzero function $f\in F_{d\cdot P\cap\Z^2}$ such that $d^{-1}\cdot\sm(f)\in U$. Clearly, $l_{P,f}\leq d$. Also, $f$ cannot be a unit in $F_{\Z^2}$ because that would imply $\sm(f)=(0,0)$. Let $f=f_1\ast\cdots\ast f_k$ be a factorization into irreducible elements of $F_{\Z^2}$. We have $\sm(f) = \sum_i \sm(f_i)$ by (\ref{smast}) and $l_{P,f} = \sum_i l_{P,f_i}$ by (\ref{NPast}) and \Cref{largest_triangle}. Hence,
\[
U \ni d^{-1}\cdot\sm(f) = \sum_i \frac{l_{P,f_i}}{d}\cdot \frac{\sm(f_i)}{l_{P,f_i}} + \left(1-\sum_i\frac{l_{P,f_i}}{d}\right)\cdot(0,0),
\]
which lies in $\hull(Z_P)$. We have shown that every open neighborhood $U$ of $e$ intersects $\hull(Z_P)$, so indeed $e\in\overline{\hull(Z_P)}$.
\qedhere
\end{description}
\end{proof}

\begin{theorem}\label{main_triangle_thm}
Let $P\subset\R^2$ be a compact triangle.
\begin{enumerate}[label=\alph*)]
\item\label{triangle_case1}
Let $f$ be an irreducible element of $F_{\Z^2}$ and $e=(e_1,e_2) = l_{P,f}^{-1}\cdot\sm(f)$. If $w:=|e|\geq \sqrt{2\vol(P)}$, then
\[
S_P = \hull(\{(0,0), (v,0), (e_1,e_2), (0,v)\})
\]
with $v:=2\vol(P)/w$.
\item\label{triangle_case2}
If there is no such irreducible element $f$, then
\[
S_P = \hull(\{(0,0),(v,0),(0,v)\})
\]
with $v:=w:=\sqrt{2\vol(P)}$.
\end{enumerate}
\end{theorem}

Note that the numbers $v$ and $w$ in \Cref{main_triangle_thm} agree with the numbers $v_P$ and $w_P$ in \Cref{vw_converge} and we have $v_Pw_P=2\vol(P)$ as claimed in the abstract.

Also, note that any function $f$ satisfying $l_{P,f}^{-1}\cdot|\sm(f)|\geq\sqrt{2\vol(P)}$ is large:
\[
|\sm(f)| \geq \sqrt{2\vol(l_{P,f}\cdot P)} \geq \sqrt{2\vol(\NP(f))}.
\]

\begin{proof}\ 
\begin{enumerate}[label=\alph*)]
\item There is no other (nonassociated) irreducible element $f'$ of $F_{\Z^2}$ with $l_{P,f'}^{-1}\cdot |\sm(f')|>v$ because \Cref{relatively_prime_inequality} says that
\begin{align*}
|\sm(f)||\sm(f')|
&\leq 2\vol(\NP(f),\NP(f')) \\
&\leq 2\vol(l_{P,f}\cdot P, l_{P,f'}\cdot P) \\
&= 2l_{P,f}l_{P,f'}\vol(P).
\end{align*}
By \Cref{triangle_spanned}, it follows that
\[
S_P \subseteq \hull(\{(0,0), (v,0), (e_1,e_2), (0,v)\}).
\]
But both sides are compact sets of volume $\vol(P)$.
\item By \Cref{triangle_spanned}, we have
\[
S_P \subseteq \hull(\{(0,0),(v,0),(0,v)\})
\]
and again both sides have the same volume $\vol(P)$.
\qedhere
\end{enumerate}
\end{proof}

In case \ref{triangle_case1}, the polygon $S_P$ is either a quadrilateral (if $e_1,e_2>0$) or a triangle, as in \Cref{shapes}.

Both shapes can occur:

\begin{example}
If $P$ is the triangle spanned by $(0,0)$, $(p,0)$, and $(0,q)$, with $p,q>0$, then $S_P=P$ according to \Cref{cont_lower}.
\end{example}

\begin{example}\label{example_quad}
Let $P$ be the triangle spanned by $(0,0)$, $(4,2)$, and $(2,3)$. Consider the function $f\in F_{P\cap\Z^2}$ with values as in \Cref{example_quad_pic}. One can check that $\sm(f) = (2,1)$. Furthermore, $f$ is irreducible due to (\ref{NPast}) because there are no two lattice polygons consisting of more than one point whose Minkowski sum is $\NP(f)=P$. It follows that $S_P$ is the quadrilateral spanned by $(0,0)$, $(\frac83,0)$, $(2,1)$, and $(0,\frac83)$. (See also \Cref{triangle_approx}.)
\end{example}
\begin{figure}[ht]
\begin{tikzpicture}[radius=2pt,below]
\draw[fill=black!10] (0,0) -- (4,2) -- (2,3) -- cycle;
\fill (0,0) circle node {$-1$};
\fill (1,1) circle node {$4$};
\fill (2,1) circle node {$-1$};
\fill (2,2) circle node {$-6$};
\fill (3,2) circle node {$4$};
\fill (4,2) circle node {$-1$};
\fill (2,3) circle node {$1$};
\end{tikzpicture}
\caption{}\label{example_quad_pic}
\end{figure}

\begin{example}
Let $P\subset\R^2$ be a compact triangle and let
\[
w := \max_{y\in\R} \vol(\{x\in\R \mid (x,y)\in P\})
\]
be the maximal length of a horizontal line segment contained in $P$. If $w \geq \sqrt{2\vol(P)}$, then
\[
S_P = \hull(\{(0,0), (w,0), (0,v)\}),
\]
where
\[
v := 2\vol(P)/w = \max\{y\mid (x,y)\in P\} - \min\{y\mid (x,y)\in P\}
\]
is the vertical height of $P$.
\end{example}
\begin{proof}
Consider the function $f\in F_{\{0,0\},\{1,0\}}$ with $f(0,0)=-1$ and $f(1,0)=1$. It is irreducible and $\sm(f)=(1,0)$. Moreover, $l_{P,f} = w^{-1}$. The claim then follows from \Cref{main_triangle_thm}.
\end{proof}

We end this section with some questions that appear to be open.

\begin{question}
Does case \ref{triangle_case2} occur for any compact triangle~$P$?
\end{question}

\begin{remark}
A related open question is whether the Seshadri constant of a weighted projective plane $\vP^{a,b,c}$ is always rational. The corresponding triangle $P$ would need to be in case \ref{triangle_case2}.
\end{remark}

\begin{remark}
Let $\mathcal T$ be the set of compact triangles $P\subset\R^2$ with the natural topology. The triangles $P$ such that there is an irreducible element $f$ of $F_{\Z^2}$ with $l_{P,f}^{-1}\cdot |\sm(f)| > \sqrt{2\vol(P)}$ form an open subset $\mathcal U$ of $\mathcal T$. (Any particular function $f$ works for an open set of triangles.)
\end{remark}

\begin{question}
Is this open subset $\mathcal U$ dense in $\mathcal T$?
\end{question}

\begin{remark}
If $\mathcal U$ is not dense in $\mathcal T$, there is a triangle $P\notin \mathcal U$ whose corners have rational coordinates but such that $\sqrt{2\vol(P)}$ is irrational. This triangle would provide an example for case \ref{triangle_case2}: If there were an irreducible element $f$ of $F_{\Z^2}$ with $l_{P,f}^{-1}\cdot |\sm(f)| = \sqrt{2\vol(P)}$, then $l_{P,f}$ would have to be irrational. This is impossible as both $P$ and $\NP(f)$ are polygons with rational coordinates.
\end{remark}

\begin{question}
What can you say about $S_P$ when $P$ is a polygon with more than three sides? For example, is $S_P$ always a polygon?
\end{question}

\begin{question}
Can these results be generalized to higher dimensions $n$? Is $S_P$ a polytope when $P$ is an $n$-dimensional simplex? Is $S_P$ a polytope when $P$ is any polytope?
\end{question}

\section{Irreducible functions}\label{irreducible_section}

We again let $n=2$ and assume that $\leq$ is the degree-lexicographic monomial order with $X_1<X_2$.

Let $\mathcal F$ be the set of large irreducible elements $f$ of $F_{\Z^2}$ and let
\[
\mathcal P = \{\NP(f) \mid f\in\mathcal F\}.
\]
We have seen that each $P\in\mathcal P$ arises from exactly one function $f$ up to convolution by a unit.

Composing any $f\in F_{\Z^2}$ with an affine linear transformation of $\R^2$ preserving $\Z^2$ doesn't change irreducibility. Since linear transformations don't change degrees of polynomials, we also have $|\sm(f)|=|\sm(f\circ T)|$. Hence, $\mathcal P$ is invariant under affine linear transformations preserving $\Z^2$. There are still infinitely many orbits under the action of this group:

\begin{theorem}\label{infinitely_many}
There are infinitely many triangles $P\in\mathcal P$ modulo affine linear transformations preserving $\Z^2$.
\end{theorem}
\begin{proof}
For $r\geq1$, let $P_r$ be the triangle spanned by the points $(0,0)$, $(r,0)$, and $(-1,r+2)$. We have $2\vol(P_r)=r(r+2)<(r+1)^2$. Let $Q_r\subset P_r$ be the triangle spanned by $(0,0)$, $(r,0)$, and $(0,r)$. (See \Cref{pic_PrQr}.) We have
\[
P_r\cap\Z^2 = (Q_r\cap\Z^2) \cup\{(-1,r+2)\}.
\]
By \Cref{lower}, we know that
\[
E_{P_r\cap\Z^2} \supseteq E_{Q_r\cap\Z^2} = Q_r\cap\N^2.
\]
In particular, the smallest leading monomial of any nonzero function vanishing on $Q_r\cap\Z^2$ is $X_1^{r+1}$. Up to scalar multiplication, there is then only one such polynomial, namely $\prod_{i=0}^r (X_1-i)$. This polynomial doesn't vanish at $(-1,r+2)$, so there is no polynomial with leading monomial $X_1^{r+1}$ vanishing on $P_r\cap\Z^2$. In other words, $(r+1,0)\in E_{P_r\cap\Z^2}$. Then, $E_{P_r\cap\Z^2} = (Q_r\cap\N^2)\cup\{(r+1,0)\}$. (See \Cref{pic_EPr}.)

Let $f_r$ be a function on $P_r\cap\Z^2$ with $\sm(f_r) = (r+1,0)$. It is up to scalar multiplication the only function $g$ on $P_r\cap\Z^2$ with $|\sm(g)| = r+1$. Since $(r+1,0)\notin E_{Q_r\cap\Z^2}$, we must have $(-1,r+2)\in\supp(f_r)$.

Similarly, since
\[
P_r\cap\Z^2\setminus\{(r,0)\} \subseteq \{-1,\dots,r-1\}\times\{0,\dots,r+2\},
\]
\Cref{combinatorial_nullstellensatz} implies that
\[
E_{P_r\cap\Z^2\setminus\{(r,0)\}} \subseteq \{0,\dots,r\}\times\{0,\dots,r+2\},
\]
so $(r+1,0)\notin E_{P_r\cap\Z^2\setminus\{(r,0)\}}$ (in fact, $E_{P_r\cap\Z^2\setminus\{(r,0)\}}=Q_r\cap\N^2$) and therefore $(r,0)\in\supp(f_r)$.

Consider the affine linear transformation $(x,y)\mapsto(r-x-y,y)$. It preserves the triangle $P_r$, but swaps the two corners $(0,0)$ and $(r,0)$. The transformation preserves $|\sm(f)|$ for any $f\in F_{\Z^2}$ and $f_r$ is up to scalar multiplication the only function with $|\sm(f_r)| = r+1$. Hence, the composition of $f_r$ with this transformation is a multiple of $f_r$. Since we already showed $(r,0)\in\supp(f_r)$, we must also have $(0,0)\in\supp(f_r)$. Hence, $\supp(f_r)$ in fact contains all three corners of $P_r$, so $\NP(f_r) = P_r$.

The polygon $P_r$ cannot be written as the Minkowski sum of two lattice polygons consisting of more than one point. Hence, $f_r$ is irreducible. We have already shown that $|\sm(f_r)|=r+1>\sqrt{2\vol(P_r)}$, so $f_r\in\mathcal F$ and $P_r\in\mathcal P$.

No two of the polygons $P_r$ for $r\geq1$ are equivalent up to an affine linear transformation preserving $\Z^2$ because they all have a different area.
\end{proof}

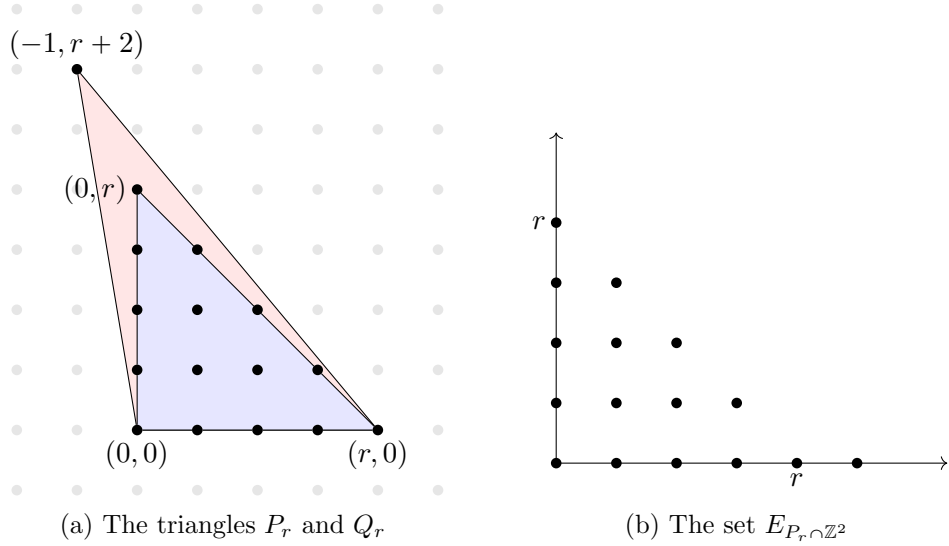
\begin{figure}[ht]
\begin{subfigure}[b]{0.47\textwidth}
\begin{tikzpicture}[radius=2/0.8pt,scale=0.8]
\tikzmath{
\r = 4;
{\draw[fill=red!10] (0,0) -- (\r,0) -- (-1,\r+2) -- cycle;};
{\draw[fill=blue!10] (0,0) -- (\r,0) -- (0,\r) -- cycle;};
int \x;
int \y;
for \x in {-2,...,\r+1}{
for \y in {-1,...,\r+3}{
{\fill[black!10] (\x,\y) circle;};
};
};
for \x in {0,...,\r}{
for \y in {0,...,\r-\x}{
{\fill (\x,\y) circle;};
};
};
{\fill (-1,\r+2) circle;};
{\draw (0,0) node[below] {$(0,0)$};};
{\draw (\r,0) node[below] {$(r,0)$};};
{\draw (-1,\r+2) node[above] {$(-1,r+2)$};};
{\draw (0,\r) node[left] {$(0,r)$};};
}
\end{tikzpicture}
\caption{The triangles $P_r$ and $Q_r$}\label{pic_PrQr}
\end{subfigure}
\hfill
\begin{subfigure}[b]{0.45\textwidth}
\begin{tikzpicture}[radius=2/0.8pt,scale=0.8]
\tikzmath{
\r = 4;
int \x;
int \y;
for \x in {0,...,\r}{
for \y in {0,...,\r-\x}{
{\fill (\x,\y) circle;};
};
};
{\fill (\r+1,0) circle;};
{\draw[->] (0,0) -- (\r+2.5,0);};
{\draw[->] (0,0) -- (0,\r+1.5);};
{\draw (\r,0) node[below] {$r$};};
{\draw (0,\r) node[left] {$r$};};
}
\end{tikzpicture}
\caption{The set $E_{P_r\cap\Z^2}$}\label{pic_EPr}
\end{subfigure}
\caption{Illustrations for the proof of \Cref{infinitely_many}}
\end{figure}

\begin{remark}
If $r<s$, then $2\vol(P_r,P_s) = (r+2)s$. Hence, the irreducible functions $f_r$ and $f_s$ satisfy equality in \Cref{relatively_prime_inequality} when $r+1=s$.
\end{remark}

There are only finitely many convex lattice polygons of positive bounded area modulo affine linear transformations preserving $\Z^2$. \Cref{irreducible_polygons} lists the elements $P$ of $\mathcal P$ of area at most $20$ up to this action, together with the corresponding value of $|\sm(f)|$.

{
\renewcommand{\arraystretch}{1.2}
\begin{longtable}{ccc}
\caption{Elements of $\mathcal P$}\label{irreducible_polygons}\\
\textbf{Corners of $P$}&$|\sm(f)|$&$2\vol(P)$\\\hline
\endfirsthead
\textbf{Corners of $P$}&$|\sm(f)|$&$2\vol(P)$\\\hline
\endhead
$(0,0),(1,0)$&$1$&$0$\\
$(0,0), (1,0), (2,3)$
&
$2$
&
$3$
\\
$(0,0), (1,0), (1,1), (-4,3)$
&
$3$
&
$8$
\\
$(0,0), (1,0), (3,8)$
&
$3$
&
$8$
\\
$(0,0), (1,0), (1,1), (-10,4)$
&
$4$
&
$15$
\\
$(0,0), (1,0), (1,1), (-9,4), (-7,3)$
&
$4$
&
$15$
\\
$(0,0), (1,0), (1,1), (-8,6)$
&
$4$
&
$15$
\\
$(0,0), (1,0), (1,2), (-3,7)$
&
$4$
&
$15$
\\
$(0,0), (1,0), (3,5), (4,10)$
&
$4$
&
$15$
\\
$(0,0), (1,0), (4,15)$
&
$4$
&
$15$
\\
$(0,0), (1,0), (1,1), (-17,5), (-7,2)$
&
$5$
&
$24$
\\
$(0,0), (1,0), (1,1), (-16,5), (-10,3)$
&
$5$
&
$24$
\\
$(0,0), (1,0), (1,1), (-11,8), (-10,7), (-3,2)$
&
$5$
&
$24$
\\
$(0,0), (1,0), (1,1), (-11,8), (-6,4)$
&
$5$
&
$24$
\\
$(0,0), (1,0), (1,1), (-9,7), (-10,7)$
&
$5$
&
$24$
\\
$(0,0), (1,0), (1,1), (-7,3), (-16,5)$
&
$5$
&
$24$
\\
$(0,0), (1,0), (1,1), (-7,6), (-11,8)$
&
$5$
&
$24$
\\
$(0,0), (1,0), (1,1), (-7,12), (-6,10), (-2,3)$
&
$5$
&
$24$
\\
$(0,0), (1,0), (1,1), (-6,11), (-7,12), (-3,5)$
&
$5$
&
$24$
\\
$(0,0), (1,0), (1,1), (-4,3), (-12,5)$
&
$5$
&
$24$
\\
$(0,0), (1,0), (1,1), (-4,9), (-2,2)$
&
$5$
&
$24$
\\
$(0,0), (1,0), (1,1), (-3,2), (-16,5), (-13,4)$
&
$5$
&
$24$
\\
$(0,0), (1,0), (1,1), (-2,5), (-4,2)$
&
$5$
&
$24$
\\
$(0,0), (1,0), (1,2), (-8,6)$
&
$5$
&
$24$
\\
$(0,0), (1,0), (1,2), (-6,5), (-7,5)$
&
$5$
&
$24$
\\
$(0,0), (1,0), (1,2), (-4,14)$
&
$5$
&
$24$
\\
$(0,0), (1,0), (1,3), (-4,9)$
&
$5$
&
$24$
\\
$(0,0), (1,0), (2,3), (-3,6)$
&
$5$
&
$24$
\\
$(0,0), (1,0), (3,3), (5,12)$
&
$5$
&
$24$
\\
$(0,0), (1,0), (5,24)$
&
$5$
&
$24$
\\
$(0,0), (1,0), (10,24)$
&
$5$
&
$24$
\\
$(0,0), (1,0), (1,1), (-15,11), (-2,1)$
&
$6$
&
$34$
\\
$(0,0), (1,0), (2,2), (-5,5), (-8,6), (-7,5)$
&
$6$
&
$34$
\\
$(0,0), (1,0), (2,2), (0,5), (-4,10), (-3,7)$
&
$6$
&
$34$
\\
$(0,0), (1,0), (1,1), (-25,6), (-13,3)$
&
$6$
&
$35$
\\
$(0,0), (1,0), (1,1), (-21,9), (-19,8), (-12,5)$
&
$6$
&
$35$
\\
$(0,0), (1,0), (1,1), (-20,13), (-17,11)$
&
$6$
&
$35$
\\
$(0,0), (1,0), (1,1), (-14,4), (-25,6)$
&
$6$
&
$35$
\\
$(0,0), (1,0), (1,1), (-14,10), (-8,5)$
&
$6$
&
$35$
\\
$(0,0), (1,0), (1,1), (-13,19), (-9,13)$
&
$6$
&
$35$
\\
$(0,0), (1,0), (1,1), (-10,4), (-20,6)$
&
$6$
&
$35$
\\
$(0,0), (1,0), (1,1), (-9,3), (-25,6), (-21,5)$
&
$6$
&
$35$
\\
$(0,0), (1,0), (1,1), (-9,15), (-8,13), (-3,4)$
&
$6$
&
$35$
\\
$(0,0), (1,0), (1,1), (-9,15), (-3,4), (-1,1)$
&
$6$
&
$35$
\\
$(0,0), (1,0), (1,1), (-8,6), (-6,2)$
&
$6$
&
$35$
\\
$(0,0), (1,0), (1,1), (-8,14), (-9,15), (-1,1)$
&
$6$
&
$35$
\\
$(0,0), (1,0), (1,1), (-8,22), (-2,5)$
&
$6$
&
$35$
\\
$(0,0), (1,0), (1,1), (-7,6), (-7,3)$
&
$6$
&
$35$
\\
$(0,0), (1,0), (1,1), (-6,4), (-15,6)$
&
$6$
&
$35$
\\
$(0,0), (1,0), (1,1), (-5,21), (-4,16), (-1,3)$
&
$6$
&
$35$
\\
$(0,0), (1,0), (1,1), (-5,21), (-2,7), (-1,3)$
&
$6$
&
$35$
\\
$(0,0), (1,0), (1,1), (-5,22), (-1,3)$
&
$6$
&
$35$
\\
$(0,0), (1,0), (1,1), (-5,27), (-1,5)$
&
$6$
&
$35$
\\
$(0,0), (1,0), (1,1), (-4,18), (-5,21), (-1,3)$
&
$6$
&
$35$
\\
$(0,0), (1,0), (1,1), (-2,2), (-20,6), (-7,2)$
&
$6$
&
$35$
\\
$(0,0), (1,0), (1,2), (-12,9)$
&
$6$
&
$35$
\\
$(0,0), (1,0), (1,2), (-7,13), (-8,14)$
&
$6$
&
$35$
\\
$(0,0), (1,0), (1,2), (-5,16), (-2,5)$
&
$6$
&
$35$
\\
$(0,0), (1,0), (1,2), (-5,21), (-2,8)$
&
$6$
&
$35$
\\
$(0,0), (1,0), (1,2), (-5,23)$
&
$6$
&
$35$
\\
$(0,0), (1,0), (1,2), (-4,9), (-8,14)$
&
$6$
&
$35$
\\
$(0,0), (1,0), (1,2), (-1,5), (-8,14)$
&
$6$
&
$35$
\\
$(0,0), (1,0), (1,3), (-5,17)$
&
$6$
&
$35$
\\
$(0,0), (1,0), (1,4), (-5,11)$
&
$6$
&
$35$
\\
$(0,0), (1,0), (3,3), (6,14), (5,13)$
&
$6$
&
$35$
\\
$(0,0), (1,0), (3,5), (6,20)$
&
$6$
&
$35$
\\
$(0,0), (1,0), (4,7), (12,28)$
&
$6$
&
$35$
\\
$(0,0), (1,0), (5,14), (6,21)$
&
$6$
&
$35$
\\
$(0,0), (1,0), (6,35)$
&
$6$
&
$35$
\\
\end{longtable}
}

\printbibliography

\end{document}